\documentclass[12pt]{amsart}

\usepackage{mathrsfs, amssymb,amsthm, amsfonts, amsmath}
\usepackage[all]{xy}
\usepackage{upgreek}
\usepackage{amsmath}

\usepackage{tabularx,multirow,makecell}
\newcolumntype{Y}{>{\centering\arraybackslash}X}
\usepackage{setspace}
\newtheorem{theorem}[equation]{Theorem}
\newtheorem*{theorem*}{Theorem}
\newtheorem{lemma}[equation]{Lemma}
\newtheorem{corollary}[equation]{Corollary}

\theoremstyle{definition}

\newtheorem{definition}[equation]{Definition}
\newtheorem*{definition*}{Definition}

\theoremstyle{remark}
\newtheorem{remark}[equation]{Remark}

\makeatletter\@addtoreset{equation}{section}
\makeatother

\usepackage{textcomp}
\usepackage{hyperref}

\newcommand{\CC}{\mathbb{C}}
\newcommand{\RR}{\mathbb{R}}
\newcommand{\QQ}{\mathbb{Q}}

\newcommand{\ZZ}{\mathbb{Z}}

\newcommand{\HHH}{{\mathscr{H}}}
\newcommand{\KKK}{{\mathscr{K}}}

\newcommand{\chit}{\chi_{\mathrm{top}}}
\newcommand{\ad}{\mathrm{a}}
\newcommand{\cc}{\mathrm{c}}
\newcommand{\bb}{\mathrm{b}}
\newcommand{\hh}{\mathrm{h}}

\newcommand{\Aut}{\operatorname{Aut}}
\newcommand{\GL}{\operatorname{GL}}
\newcommand{\SL}{\operatorname{SL}}

\newcommand{\z}{\operatorname{z}}
\newcommand{\rad}{\operatorname{rad}}

\newcommand{\rk}{\operatorname{rk}}

\def \ge {\geqslant}
\def \le {\leqslant}

\date{}

\title{Automorphism groups of Inoue and Kodaira surfaces}

\author{Yuri Prokhorov}
\author{Constantin Shramov}

\address{
Steklov Mathematical Institute of Russian Academy of Sciences, 8 Gubkina st.,
Moscow, 119991, Russia
\newline
National Research University Higher School of Economics, Laboratory of Algebraic Geometry, 6 Usacheva str., Moscow, 119048, Russia
}

\email{prokhoro@mi.ras.ru}
\email{costya.shramov@gmail.com}

\thanks{
This work is supported by the Russian Science Foundation under grant \textnumero 18-11-00121.}

\begin{document}

\begin{abstract}
We prove that automorphism groups of Inoue and primary Kodaira surfaces are Jordan.
\end{abstract}

\maketitle
\tableofcontents

\section{Introduction}

It sometimes happens that groups of geometric origin are complicated and difficult to study.
On the other hand, in many situations they are easier to access on the level of their finite
subgroups. In particular, the following notion appeared to be very useful.

\begin{definition}[{see \cite[Definition~2.1]{Popov2011}}]
\label{definition:Jordan}
A group~$\Gamma$ is called \emph{Jordan}
(alternatively, we say
that~$\Gamma$ \emph{has Jordan property})
if there is a constant~$J$ such that
for every finite subgroup~\mbox{$G\subset\Gamma$} there exists
a normal abelian subgroup $A\subset G$ of index at most~$J$.
\end{definition}

An old result due to C.\,Jordan says that the group $\GL_n(\CC)$ is
Jordan (see e.g.~\mbox{\cite[Theorem~36.13]{Curtis-Reiner-1962}}).
S.\,Meng and D.-Q.\,Zhang proved in \cite{MengZhang} that an automorphism group
of any projective variety is Jordan. T.\,Bandman and Yu.\,Zarhin proved
a similar result for automorphism groups of quasi-projective surfaces
in~\cite{BandmanZarhin2015},
and also in some particular cases in arbitrary dimension in~\cite{BandmanZarhin2017}.
J.\,H.\,Kim in \cite{Kim} generalized the result of~\cite{MengZhang}
to the case of automorphism groups of compact K\"ahler varieties.
There are many further results on the Jordan property
for diffeomorphism groups of smooth compact manifolds, see
\cite{Popov-Diff}, \cite{CsikosPyberSzabo}, \cite{Riera2016}, \cite{Riera-Spheres}, \cite{Riera-OddCohomology},
\cite{Riera-Symp}, \cite{Riera-HamSymp}, \cite{Ye2017}, and references therein.
There are also numerous results concerning Jordan property for groups
of birational automorphisms of projective varieties, see
\cite{Serre2009}, \cite{Popov2011}, \cite{ProkhorovShramov-RC},
\cite{Prokhorov-Shramov-2013}, \cite{ProkhorovShramov-dim3}, \cite{BandmanZarhin2015a},
\cite{Prokhorov-Shramov-JCr3}, \cite{Prokhorov-Shramov-p-groups}, \cite{Yasinsky2016a}.

Our goal is to study finite subgroups of automorphism groups
of \emph{compact complex surfaces}, that is, connected compact complex manifolds of dimension~$2$.
Recall that such a surface is called \emph{minimal} if it does not contain smooth rational
curves with self-intersection equal to~$(-1)$.
There is a classification of minimal compact complex surfaces, known as
Enriques--Kodaira classification, see e.g.~\cite[Chapter~VI]{BHPV-2004}.
Two important classes of such surfaces are \emph{Inoue surfaces} (see \cite{Inoue1974})
and \emph{primary Kodaira surfaces} (see~\mbox{\cite[\S6]{Kodaira-structure-1}},
\cite[\S\,V.5]{BHPV-2004}). The former are minimal compact complex surfaces of Kodaira dimension $-\infty$,
vanishing algebraic dimension, vanishing second Betti number, and containing no curves.
The latter are minimal compact complex surfaces with trivial canonical class
and first Betti number equal to $3$.
Both of these classes do not contain projective or K\"ahler surfaces,
so the methods of \cite{MengZhang} and \cite{Kim} do not provide an approach to
their automorphism groups.

The main goal of this paper is to prove the following.

\begin{theorem}
\label{theorem:main}
Let $X$ be either an Inoue surface or a primary Kodaira surface.
Then the automorphism group of $X$ is Jordan.
\end{theorem}

One feature of our proof that we find interesting to mention
is that Inoue and Kodaira surfaces are treated by literally the same method
which is based on the fact that they are diffeomorphic to solvmanifolds
(cf.~\cite[Theorem~1]{Hasegawa}),
and for which we never met a proper analog in the projective situation.
It is possible that one can generalize this approach to higher dimensional
solvmanifolds.

The plan of the paper is as follows.
In~\S\ref{section:matrices} we collect some auxiliary facts.
In \S\S\ref{section:lattices-semi-direct}--\ref{section:Heisenberg-direct} we establish
assertions about discrete groups (more precisely, Wang groups, see~\cite{Hasegawa})
that will be used in our analysis of automorphism
groups of surfaces.
In~\S\ref{section:Inoue} we prove Jordan property for automorphism groups of Inoue surfaces.
In~\S\ref{section:Kodaira} we do the same for automorphism groups of primary Kodaira surfaces.
Our main result, Theorem~\ref{theorem:main}, is just a composition of Lemma~\ref{corollary:Inoue} and~\ref{lemma:primary-Kodaira-surface}.

\smallskip
\textbf{Notation.}
For every group~$\Gamma$ we denote by $\z(\Gamma)$ the center of $\Gamma$, and
for a subgroup~\mbox{$\Gamma'\subset\Gamma$} we denote by $\z(\Gamma',\Gamma)$ the centralizer of $\Gamma'$ in $\Gamma$.
Given a compact complex surface~$X$,
we denote by~$\KKK_X$ its canonical line bundle, by~$\bb_i(X)$ its $i$-th
Betti number, by $\chit(X)$ the topological Euler characteristic,
by~$\varkappa(X)$ the Kodaira dimension, and by~$\ad(X)$ the algebraic dimension of~$X$
By $\hh^{p,q}(X)$ we denote the Hodge numbers~\mbox{$\hh^{p,q}=\dim H^q(X,\Omega^p_X)$},
where $\Omega^p_X$ is the sheaf of holomorphic $p$-forms on~$X$.

\smallskip
\textbf{Acknowledgements.}
We are grateful to M.\,Finkelberg, S.\,Gorchinskiy, S.\,Nemirovski, D.\,Osipov, E.\,Rousseau, and M.\,Verbitsky for useful discussions.

\section{Preliminaries}
\label{section:matrices}

In this section we collect several auxiliary facts that will be used in the remaining parts of the paper.

We say that a group $\Gamma$
\emph{has bounded finite subgroups}
if there exists a constant $B=B(\Gamma)$ such that
for any finite subgroup
$G\subset\Gamma$ one has $|G|\leqslant B$. If this is not the case,
we say that $\Gamma$ has \emph{unbounded finite subgroups}.
The following result is due to
H.\,Minkowski (see e.\,g.~\cite[Theorem~1]{Serre2007}).

\begin{theorem}
\label{theorem:Minkowski}
For every $n$ the group $\GL_n(\QQ)$ has bounded finite subgroups.
\end{theorem}

\begin{lemma}
\label{lemma:group-theory}
Let
$$
1\longrightarrow\Gamma'\longrightarrow\Gamma\longrightarrow\Gamma''
$$
be an exact sequence of groups.
Suppose that $\Gamma'$ is Jordan and $\Gamma''$ has bounded finite
subgroups. Then~$\Gamma$ is Jordan.
\end{lemma}
\begin{proof}
Obvious.
\end{proof}

The following facts from number theory are well known to experts; we include their proofs for the reader's convenience.

\begin{lemma}\label{lemma:unit}
The following assertions hold.
\begin{enumerate}
\item
Let $\alpha$ be an algebraic integer such that for every Galois
conjugate $\alpha'$ of $\alpha$ one has~\mbox{$|\alpha'|=1$}.
Then $\alpha$ is a root of unity.

\item
Let $n$ be a positive integer. Then there exists a constant $\varepsilon=\varepsilon(n)$
with the following property: if an algebraic integer $\alpha$ of degree $n$ is such that for every Galois
conjugate $\alpha'$ of $\alpha$ one has $1-\varepsilon<|\alpha'|<1+\varepsilon$, then $\alpha$ is a root of unity.
\end{enumerate}
\end{lemma}

\begin{proof}
To prove assertion~(i), fix an embedding $\QQ(\alpha)\subset\CC$.
Then $\bar{\alpha}$ is a root of the minimal polynomial of $\alpha$ over~$\QQ$.
Hence $\bar{\alpha}$ is an algebraic integer. On the other hand, one has
$\bar{\alpha}=\alpha^{-1}$. Since both~$\alpha$ and~$\alpha^{-1}$ are algebraic integers,
we conclude that all non-archimedean
valuations of~$\alpha$ equal~$1$. At the
same time all archimedean valuations of~$\alpha$ equal~$1$ by assumption.
Therefore, assertion~(i) follows from~\mbox{\cite[Lemma~II.18.2]{CasselsFrolich}}.

Now consider an algebraic integer $\alpha$ of degree $n$ such that all its Galois
conjugates have absolute values less than, say, $2$. The absolute values of the coefficients of its
minimal polynomial are bounded by some constant $C=C(n)$; for instance, one can take~\mbox{$C=2^{2n}$}.
Consider the set of polynomials with integer coefficients
$$
\mathcal{Q}=\bigl\{Q=x^n+a_{n-1}x^{n-1}+\ldots+a_0 \,\bigl|\, |a_i|\le C, \text{\ and $Q$ is irreducible} \bigr\}.
$$
The set $\mathcal{Q}$ is finite. Put
$$
\Pi=\bigl\{(x_1,\ldots,x_n) \,\bigl|\, x_1,\ldots,x_n \text{\ are different roots of some polynomial\ }
Q\in\mathcal{Q} \bigr\}\subset\CC^n.
$$
Then $\Pi$ is a finite subset of $\CC^n$; furthermore,
all algebraic integers of degree $n$ such that all their Galois conjugates have absolute values at most $2$
appear as coordinates
of the points of~$\Pi$. There is a number $\mu=\mu(n)$ such that for every point
$P=(\alpha_1,\ldots,\alpha_n)\in \Pi$
the inequalities
$1-\mu<|\alpha_i|<1+\mu$ for all $i$
imply that $|\alpha_i|=1$ for all $i$. In the latter case
$\alpha_i$ are roots of unity by assertion~(i). Thus it remains to put
$\varepsilon=\min(\mu,1)$ to prove assertion~(ii).
\end{proof}

\begin{lemma}\label{lemma:matrix-root}
Let $M\in\GL_n(\ZZ)$ be a matrix. Suppose that for every $C$ there is an integer~\mbox{$k>C$} such that there exists a
matrix $R_k\in\GL_n(\ZZ)$ with $R_k^{\, k}=M$. Then all eigen-values of $M$ are roots of unity.
\end{lemma}
\begin{proof}
Let $\lambda_k$ be an eigen-value of $R_k$. Then $\lambda_k$ is an algebraic integer
of degree at most~$n$, because it is a root of the characteristic
polynomial of the matrix $R_k$. Moreover, $\lambda_k^{\, k}$ is an eigen-value of $M$. This means
that
\begin{equation}\label{eq:strip}
\sqrt[k]{l_{min}}\le |\lambda_k|\le \sqrt[k]{l_{max}},
\end{equation}
where $l_{min}$ and $l_{max}$ are the minimal
and the maximal absolute values of the eigen-values of the matrix $M$, respectively.
Both of the above bounds converge to $1$ when $k$ goes to infinity.
All Galois conjugates of $\lambda_k$ are also eigen-values of $R_k$, hence
the inequality~\eqref{eq:strip} holds for them as well.
Therefore, for $k$ large enough all eigen-values of $R_k$ are roots of unity by Lemma~\ref{lemma:unit}(ii),
and thus so are the eigen-values of $M$.
\end{proof}

\section{Lattices and semi-direct products}
\label{section:lattices-semi-direct}

Consider the groups~\mbox{$\Gamma_0\cong \ZZ^3$} and~\mbox{$\Gamma_1\cong \ZZ$}.
Here and below we will use multiplicative notation for the operations in all arising groups, in particular,
in~$\Gamma_0$ and $\Gamma_1$.
Let $\upgamma$ be a generator of $\Gamma_1$. Fix a basis $\{\updelta_1,\updelta_2,\updelta_3\}$ in~$\Gamma_0$.
Then~\mbox{$\operatorname{End}(\Gamma_0)$}
can be identified with~\mbox{$\operatorname{Mat}_{3\times 3}(\ZZ)$}, and so for any integral
$3\times 3$-mat\-rix~\mbox{$M=(m_{j,i})$} one can define its action on~$\Gamma_0$ via
\[
M (\updelta_i)= \updelta_1^{m_{1,i}}\updelta_2^{m_{2,i}}\updelta_3^{m_{3,i}}.
\]
If $M\in \GL_3(\ZZ)$,
this defines a semi-direct product $\Gamma=\Gamma_0 \rtimes \Gamma_1$.

The following facts are easy exercises in group theory.

\begin{lemma}\label{lemma:lattice-exercise}
Suppose that the matrix $M$ does not have eigen-values equal to~$1$. Then the following assertions hold:
\begin{enumerate}
\item\label{lemma:lattice-exercisei}
$[\Gamma,\Gamma]=\operatorname{Im}(M-\mathrm{Id})\subset\Gamma_0$ is a free abelian subgroup of rank $3$;

\item\label{lemma:lattice-exerciseii}
one has $\Gamma_0=\z([\Gamma,\Gamma],\Gamma)$; in particular, $\Gamma_0$ is a characteristic subgroup of $\Gamma$;

\item\label{lemma:lattice-exerciseiii} the center $\z(\Gamma)$ is trivial.
\end{enumerate}
\end{lemma}

\begin{proof}
The subgroup $[\Gamma,\Gamma]\subset\Gamma_0$ is generated by the commutators
$$
[\upgamma,\updelta_i]=\upgamma\updelta_i\upgamma^{-1}\updelta_i^{-1}=
\updelta_1^{m_{1,i}}\updelta_2^{m_{2,i}}\updelta_3^{m_{3,i}}\updelta_i^{-1},
$$
where $i=1,2,3$.
Therefore, it can be identified with the sublattice in $\Gamma_0$ that is the image of
the operator $M-\mathrm{Id}$. Since the latter operator is non-degenerate
by assumption, we conclude that $[\Gamma,\Gamma]$ is a free abelian group of rank~$3$.
This proves assertion~\ref{lemma:lattice-exercisei}.

Since $\Gamma_0$ has no torsion, by assertion~\ref{lemma:lattice-exercisei}
we have $\z([\Gamma,\Gamma],\Gamma)= \z(\Gamma_0,\Gamma)\supset \Gamma_0$.
Thus, to prove assertion~\ref{lemma:lattice-exerciseii} it is
enough to show that no non-trivial power of $\upgamma$ is contained in
the above centralizer.
But if $\upgamma^k\in\z(\Gamma_0,\Gamma)$, then $(M-\mathrm{Id})^k=0$.
This contradicts our assumptions, which proves assertion~\ref{lemma:lattice-exerciseii}.

By assertion~\ref{lemma:lattice-exerciseii}, we have $\z(\Gamma)\subset \Gamma_0$.
Considering the commutators with $\upgamma$, we see that~\mbox{$\z(\Gamma)\subset \operatorname{Ker}(M-\mathrm{Id})$.}
Thus, we have $\z(\Gamma)=\{1\}$,
which proves assertion~\ref{lemma:lattice-exerciseiii}.
\end{proof}

It also appears that one can easily describe all normal subgroups of finite index in $\Gamma$
(and actually do it in a slightly more general setting).

\begin{lemma}\label{lemma:Inoue-normal-subgroup}
Let $\Delta_0$ be an arbitrary group (and $\Gamma_1\cong\ZZ$ as before
be a cyclic group generated by an element $\upgamma$).
Consider a semi-direct product $\Delta=\Delta_0\rtimes\Gamma_1$.
Let $\Delta'\subset \Delta$ be a normal subgroup of finite index. Then
\begin{enumerate}
\item\label{lemma:Inoue-normal-subgroupi}
$\Delta'=\Delta_0' \rtimes \Gamma_1'$, where
$\Delta_0'= \Delta'\cap \Delta_0$,
and $\Gamma_1'\cong\ZZ$ is generated by $\upgamma^k\updelta'$ for some
positive integer $k$ and $\updelta'\in \Delta_0$;

\item\label{lemma:Inoue-normal-subgroupii}
$\Delta/\Delta'$ has a normal subgroup of index $k$
isomorphic to $\Delta_0/\Delta_0'$.
\end{enumerate}
\end{lemma}

\begin{proof}
The subgroup $\Delta'/\Delta_0'$ has finite index in $\Gamma_1\cong\ZZ$. Thus it is generated
by the image of an element~$\upgamma^k$ for some positive integer~$k$. Choose a preimage $\uptheta=\upgamma^k\updelta$ in $\Delta'$,
where~\mbox{$\updelta\in\Delta_0'$}.
Let $\Gamma_1'$ be the subgroup of $\Delta'$ generated by $\uptheta$.
Then $\Delta'$ is generated by its subgroups $\Delta_0'$ and~\mbox{$\Gamma_1'\cong\ZZ$}. This proves assertion~\ref{lemma:Inoue-normal-subgroupi}.

Consider the image $\bar{\Delta}_0$ of $\Delta_0$ in the quotient group $\Delta/\Delta'$. It is isomorphic to $\Delta_0/\Delta_0'$. Furthermore, the quotient
$(\Delta/\Delta')/\bar{\Delta}_0$ maps isomorphically to the quotient of $\Gamma_1$ by the subgroup
generated by $\upgamma^k$. Thus, $\bar{\Delta}_0$ is a normal subgroup of index $k$ in~$\Delta/\Delta'$.
This proves assertion~\ref{lemma:Inoue-normal-subgroupi}.
\end{proof}

\begin{lemma}\label{lemma:lattices-semi-direct}
Let $\Gamma=\Gamma_0\rtimes\Gamma_1$ be a semi-direct product defined by a matrix $M$ as above.
Suppose that the matrix $M$ does not have eigen-values equal to~$1$,
and at least one of its eigen-values is not a root of unity. Then
there exists a constant $\nu=\nu(\Gamma)$ with the following property.

Let~\mbox{$\hat{\Gamma}=\hat{\Gamma}_0\rtimes\hat{\Gamma}_1$}, where~\mbox{$\hat{\Gamma}_0\cong\ZZ^3$}
and $\hat{\Gamma}_1\cong\ZZ$. Suppose that $\hat{\Gamma}$ contains~$\Gamma$ as a normal subgroup.
Then the group $G=\hat{\Gamma}/\Gamma$ is finite and has a normal abelian subgroup
of index at most~$\nu$.
\end{lemma}
\begin{proof}
The group $G$ is finite for obvious reasons.
By Lemma~\ref{lemma:lattice-exercise}(ii) we have~\mbox{$\Gamma_0=\Gamma\cap\hat{\Gamma}_0$}.
Thus by Lemma~\ref{lemma:Inoue-normal-subgroup}
there is a positive integer $k$ with the following properties:
first, the subgroup $\Gamma_1\subset\Gamma$ is generated by an element $\hat{\upgamma}^k\hat{\updelta}$, where $\hat{\upgamma}$ is
a generator of $\hat{\Gamma}_1$, and $\hat{\updelta}$ is an element of $\hat{\Gamma}_0$;
second, the group $G$ contains a normal abelian subgroup
of index~$k$.
Note that the subgroup $\Gamma_0$ is normal in $\hat{\Gamma}$,
because $\Gamma$ is normal in $\hat{\Gamma}$, while
$\Gamma_0$ is a characteristic subgroup of $\Gamma$ by Lemma~\ref{lemma:lattice-exercise}(ii).
Let $R\in\GL_3(\ZZ)$ be the matrix that defines the semi-direct product
$\hat{\Gamma}=\hat{\Gamma}_0\rtimes\hat{\Gamma}_1$.
Considering the action of the element~$\hat{\upgamma}$
on the lattice~\mbox{$\hat{\Gamma}_0\cong\ZZ^3$} and its sublattice
$\Gamma_0\cong\ZZ^3$,
we see that $R^k$ is conjugate to $M$.
Thus $k$ is bounded by some constant $\nu$ that depends only on $M$
(that is, only on $\Gamma$) by Lemma~\ref{lemma:matrix-root}.
\end{proof}

\section{Heisenberg groups}
\label{section:Heisenberg}

Let $r$ be a positive integer. Consider a group
\begin{equation}\label{eq:Heisenberg-generators}
\HHH(r)=\langle\updelta_1,\updelta_2,\updelta_3\mid [\updelta_i,\updelta_3]=1, [\updelta_1,\updelta_2]=\updelta_3^r\rangle.
\end{equation}
One can think about $\HHH(r)$ as the group of all matrices
\begin{equation*}%\label{eq:gamma}
\left(
\begin{array}{ccc}
1 & a & \frac{c}{r}\\
0 & 1 & b\\
0 & 0 & 1
\end{array}
\right)\in\GL_3(\QQ),
\end{equation*}
where $a$, $b$, and $c$ are integers.
One can choose the generators so that the element~$\updelta_1$ corresponds to $a=1$, $b=c=0$,
the element $\updelta_2$ corresponds to $a=0$, $b=1$, $c=0$,
and the element $\updelta_3$ corresponds to $a=b=0$, $c=1$.
The group $\HHH(1)$ is known as the discrete Heisenberg group.
The center $\z(\HHH(r))\cong\ZZ$ is generated by~$\updelta_3$,
while the commutator~\mbox{$[\HHH(r),\HHH(r)]$} is generated by~$\updelta_3^r$.
For the quotient group~\mbox{$\bar{\HHH}(r)=\HHH(r)/\z(\HHH(r))$} one has $\bar{\HHH}(r)\cong\ZZ^2$.

\begin{lemma}\label{lemma:Heisenberg-finite-index}
Every subgroup of finite index in $\HHH(r)$ is isomorphic to $\HHH(r')$ for some positive integer~$r'$.
Every subgroup of infinite index in $\HHH(r)$ is abelian.
\end{lemma}

\begin{proof}
Let $\Delta$ be a subgroup of $\HHH(r)$. Denote by $\Delta_{\z}$ the intersection $\Delta\cap\z(\HHH(r))$,
and let $\bar{\Delta}\cong\Delta/\Delta_{\z}$ be the image of $\Delta$ in $\bar{\HHH}(r)$.

Suppose that $\Delta$ has finite index in $\HHH(r)$.
Then $\Delta_{\z}$ has finite index in $\z(\HHH(r))$, so that it is generated by $\updelta_3^k$ for some positive integer $k$.
Furthermore, $\bar{\Delta}$ has finite index in~\mbox{$\bar{\HHH}(r)\cong\ZZ^2$}, so that $\bar{\Delta}\cong\ZZ^2$.
Choose generators $\bar{\uptheta}_i$, $i=1,2$, of the group $\bar{\Delta}$,
and let~\mbox{$\uptheta_i$, $i=1,2$}, be their preimages in $\Delta$. Then $\Delta$ is generated by $\uptheta_1$, $\uptheta_2$,
and $\updelta_3^k$. One has~\mbox{$[\uptheta_1,\uptheta_2]=\updelta_3^d$} for some integer~$d$. In particular,
we see that $d$ is divisible by $k$. This implies that $\Delta\cong\HHH(r')$ for $r'=\frac{|d|}{k}$.

Now suppose that $\Delta$ has infinite index in $\HHH(r)$. If $\Delta_{\z}$ is trivial, then $\Delta\cong\bar{\Delta}$
is abelian. So we may assume that $\Delta_{\z}$ is not trivial. This means that $\Delta_{\z}$ has finite
index in~\mbox{$\z(\HHH(r))\cong\ZZ$}, and $\bar{\Delta}$ has infinite index in $\bar{\HHH}(r)$.
The latter implies that $\bar{\Delta}\cong\ZZ$. Thus~$\Delta$ is a central extension of a cyclic group,
so it is abelian.
\end{proof}

Note that a subgroup in $\HHH(r)$ generated by $\updelta_1^a$, $\updelta_2$, and $\updelta_3^c$
is isomorphic to $\HHH\left(\frac{ar}{c}\right)$.
Hence any group $\HHH(r')$ is realized as a subgroup of a given group $\HHH(r)$.
We will be interested in properties of normal subgroups of $\HHH(r)$.

\begin{lemma}\label{lemma:Heisenberg-subgroups}
Let $\Gamma_0\subset\HHH(r)$ be a normal subgroup of finite index, and put~\mbox{$G_0=\HHH(r)/\Gamma_0$}.
Then there are integers $a_1,a_2,a_3,b_1,b_2,b_3$
with~\mbox{$a_1b_2-a_2b_1\neq 0$}, and~\mbox{$c>0$} such that
$\updelta_3^c$ generates the intersection of $\Gamma_0$ with $\z(\HHH(r))$, and
$$
\Gamma_0=\langle\updelta_1^{a_1}\updelta_2^{a_2}\updelta_3^{a_3},\
\updelta_1^{b_1}\updelta_2^{b_2}\updelta_3^{b_3},\ \updelta_3^c \rangle.
$$
Moreover, the number $c$ divides $r\gcd(a_1,a_2,b_1,b_2)$, one has
$$
\Gamma_0\cong\HHH\left(\frac{r|a_1b_2-a_2b_1|}{c}\right),
$$
and the group $G_0$ contains a normal abelian subgroup of index at most $\gcd(a_1,b_1)$.
\end{lemma}
\begin{proof}
Since $G_0$ is finite, the image $\bar{\Gamma}_0$ of $\Gamma_0$ in $\bar{\HHH}(r)$
is isomorphic to $\ZZ^2$. Choose the vectors $(a_1,a_2)$ and $(b_1,b_2)$ in $\bar{\HHH}(r)\cong\ZZ^2$
generating $\bar{\Gamma}_0$. The group $\Gamma_0$ contains the elements
$\upzeta=\updelta_1^{a_1}\updelta_2^{a_2}\updelta_3^{a_3}$ and $\upxi=\updelta_1^{b_1}\updelta_2^{b_2}\updelta_3^{b_3}$
for some integers $a_3$ and $b_3$.
The subgroup of~$\Gamma_0$ generated by $\upzeta$ and $\upxi$ maps surjectively to $\bar{\Gamma}_0$.
Hence $\Gamma_0$ is generated by $\upzeta$, $\upxi$, and the intersection $\Gamma_0\cap\z(\HHH(r))$.
The latter is a subgroup of $\z(\HHH(r))\cong\ZZ$, and thus is generated by some element of the form $\updelta_3^c$.

Since the subgroup $\Gamma_0$ is normal, we have $\updelta_3^{rb_1}=[\updelta_1,\upzeta]\in\Gamma_0$,
so that $c$ divides $ra_2$. Similarly, we see that $c$ divides $ra_1$, $rb_1$, and $rb_2$, and thus
also divides $r\gcd(a_1,a_2,b_1,b_2)$.

It is easy to compute that $[\upzeta,\upxi]=\updelta_3^{r(a_1b_2-a_2b_1)}$.
Therefore, one has
$$
\Gamma_0\cong\HHH\left(\frac{r|a_1b_2-a_2b_1|}{c}\right).
$$

Let $F$ be the subgroup of $\HHH(r)$ generated by the elements $\updelta_2$ and $\updelta_3$,
and $\breve{F}$ be its image in $G_0$. The subgroup $F$ is a normal abelian subgroup of $\HHH(r)$,
hence $\breve{F}$ is a normal abelian subgroup of $G_0$. Let $f\colon\HHH(r)\to G_0/\breve{F}$
be the natural projection. Then the group~\mbox{$G_0/\breve{F}$} is generated by~$f(\updelta_1)$,
and one has
$$
f(\updelta_1^{a_1})=f(\upzeta)=1=f(\upxi)=f(\updelta_1^{b_1}).
$$
Hence $[G_0:\breve{F}]=|G_0/\breve{F}|$ is bounded from above by (and actually equals)
the number~\mbox{$\gcd(a_1,b_1)$}.
\end{proof}

An immediate consequence of Lemma~\ref{lemma:Heisenberg-subgroups} is the following boundedness result.

\begin{corollary}\label{corollary:Heisenberg-quotients}
Let $\Gamma_0\subset\hat{\Gamma}_0$ be a normal subgroup,
where $\Gamma_0\cong\HHH(r_1)$ and $\hat{\Gamma}_0\cong\HHH(r_2)$.
Then the quotient group $G_0=\hat{\Gamma}_0/\Gamma_0$ is finite,
and $G_0$ contains a normal abelian subgroup of index at most~$r_1$.
\end{corollary}
\begin{proof}
The group $G_0$ is finite for obvious reasons.
By Lemma~\ref{lemma:Heisenberg-subgroups}
there are integers $a_1,a_2,b_1,b_2$ with $a_1b_2-a_2b_1\neq 0$, and $c>0$ such that
$c$ divides $r\gcd(a_1,a_2,b_1,b_2)$, one has
$$
r_1=\frac{r_2|a_1b_2-a_2b_1|}{c},
$$
and $G_0$ contains a normal abelian subgroup of index at most
$\gcd(a_1,b_1)$. On the other hand, one has
$$
\frac{r_2|a_1b_2-a_2b_1|}{c}\ge \frac{r_2\gcd(a_1,b_1)\gcd(a_2,b_2)}{c}\ge \frac{r_2\gcd(a_1,a_2,b_1,b_2)}{c}\gcd(a_1,b_1)\ge \gcd(a_1,b_1).
$$
\end{proof}

\section{Heisenberg groups and semi-direct products}
\label{section:Heisenberg-semi-direct}

Consider the groups~\mbox{$\Gamma_0\cong \HHH(r)$} and $\Gamma_1\cong \ZZ$.
Let $\upgamma$ be a generator of $\Gamma_1$.
Consider a semi-direct product $\Gamma=\Gamma_0 \rtimes \Gamma_1$.
The action of $\upgamma$ on $\Gamma_0$ gives rise to
its action on
\[
\bar{\Gamma}_0=\Gamma_0/\z(\Gamma_0)\cong\ZZ^2,
\]
which is given by a matrix $M\in\GL_2(\ZZ)$ if we fix a basis in
$\bar{\Gamma}_0$ (cf.~\cite{Osipov} for a detailed description of the automorphism
group of the discrete Heisenberg group).

\begin{lemma}\label{lemma:Heisenberg-semi-direct-center}
The following assertions hold.
\begin{enumerate}
\item
One has $\upgamma\updelta_3\upgamma^{-1}=\updelta_3^{\det M}$.

\item The center $\z(\Gamma)$ is trivial if and only if $\det M=-1$.

\item One has $[\Gamma,\Gamma]\subset\Gamma_0$.
\end{enumerate}
\end{lemma}
\begin{proof}
For $i=1,2$ one has
$$
\upgamma\updelta_i\upgamma^{-1}=\updelta_1^{m_{1,i}}\updelta_2^{m_{2,i}}\updelta_{3}^{p_{i}},
$$
where $M=(m_{i,j})$, and $p_i$ are some integers. Obviously, we have $\upgamma\updelta_3\upgamma^{-1}=\updelta_3^t$
for some integer $t$. Therefore
\begin{multline*}
\updelta_3^{rt}=\upgamma\updelta_3^r\upgamma^{-1}=
\upgamma\updelta_1\updelta_2\updelta_1^{-1}\updelta_2^{-1}\upgamma^{-1}=
\updelta_1^{m_{1,1}}\updelta_2^{m_{2,1}}\updelta_1^{m_{1,2}}\updelta_2^{m_{2,2}}
\updelta_2^{-m_{2,1}}\updelta_1^{-m_{1,1}}\updelta_2^{-m_{2,2}}\updelta_1^{-m_{1,2}}=\\ =
\updelta_3^{r(m_{1,1}m_{2,2}-m_{1,2}m_{2,1})}=\updelta_3^{r\det M},
\end{multline*}
which implies assertion~(i).
Assertion~(ii) easily follows from assertion~(i). Assertion~(iii) is obvious.
\end{proof}

We will need the following
notation. Let $\Upsilon$ be a group, and $\Delta$ be its subset.
Denote
$$
\rad(\Delta,\Upsilon)=\{g\in\Upsilon\mid g^k\in\Delta
\text{\ for some positive integer $k$\ }\}.
$$
If $\Delta$ is invariant with respect to
some automorphism of $\Upsilon$, then $\rad(\Delta,\Upsilon)$ is invariant
with respect to this automorphism as well.
If a group $\Upsilon$ has no torsion and $\Delta\subset\Delta'$ is a pair of
subgroups in $\Upsilon$ such that the index $[\Delta':\Delta]$ is finite,
then $\Delta'\subset\rad(\Delta,\Upsilon)$.

\begin{lemma}\label{lemma:Heisenberg-exercise}
Suppose that the matrix $M$ does not have eigen-values equal to~$1$. Then the following assertions hold:
\begin{enumerate}
\item one has
$$
[\Gamma,\Gamma]/\z([\Gamma,\Gamma])=\operatorname{Im}(M-\mathrm{Id})\subset\bar{\Gamma}_0
$$
is a free abelian group of rank $2$;

\item
$[\Gamma,\Gamma]\cong\HHH(r')$ for some $r'$, and
$[\Gamma,\Gamma]$ is a subgroup of finite index in $\Gamma_0$;

\item
$\Gamma_0=\rad([\Gamma,\Gamma],\Gamma)$; in particular, $\Gamma_0$
is a characteristic subgroup of $\Gamma$.
\end{enumerate}
\end{lemma}

\begin{proof}
The subgroup $[\Gamma,\Gamma]/\z([\Gamma,\Gamma])\subset\bar{\Gamma}_0$ is generated by the images of the
commutators~\mbox{$[\upgamma,\updelta_i]$}, where $i=1,2$.
Therefore, it can be identified with the sublattice in $\Gamma_0$ that is the image of
the operator $M-\mathrm{Id}$. Since the latter operator is non-degenerate
by assumption, we conclude that $[\Gamma,\Gamma]/\z([\Gamma,\Gamma])$ is a free abelian group of rank $2$.
This proves assertion~(i).

Since $[\Gamma,\Gamma]/\z([\Gamma,\Gamma])$ has finite index in $\bar{\Gamma}_0$ by assertion~(i), and $[\Gamma,\Gamma]$
contains the element $\updelta_3^{\det M}$ by Lemma~\ref{lemma:Heisenberg-semi-direct-center}, we conclude that
$[\Gamma,\Gamma]$ is a subgroup of finite index in $\Gamma_0$. Therefore, it is isomorphic to a group
$\HHH(r')$ for some positive integer $r'$ by Lemma~\ref{lemma:Heisenberg-finite-index}. This gives assertion~(ii).

Since $[\Gamma,\Gamma]$ has finite index in $\Gamma_0$, we see that $\Gamma_0\subset\rad([\Gamma,\Gamma],\Gamma)$.
Since $\Gamma/\Gamma_0\cong\ZZ$ has no torsion, we see that
the opposite inclusion holds as well. This gives assertion~(iii).
\end{proof}

\begin{remark}\label{remark:Heisenberg-semi-direct-no-1}
If the matrix $M$ does not have eigen-values equal to~$1$,
then the group~$\Gamma$ does not contain subgroups
isomorphic to~$\ZZ^3$.
Indeed, suppose that $\Delta\subset\Gamma$ is such a subgroup. Then both the intersection
$\Delta_{\z}=\Delta\cap\z(\Gamma)$ and the image $\bar{\Delta}_0$
of $\Delta\cap\Gamma_0$ in $\bar{\Gamma}_0$ are non-trivial. Now the contradiction is obtained by taking a commutator
of a non-trivial element of $\Delta_{\z}$ with a preimage in $\Delta$ of a non-trivial element
of $\bar{\Delta}_0$ and using the fact that the operator $M-\mathrm{Id}$ is non-degenerate.
\end{remark}

\begin{lemma}\label{lemma:Heisenberg-semi-direct}
Suppose that the eigen-values of the matrix $M$ are not roots of unity. Then
there exists a constant $\nu=\nu(\Gamma)$ with the following property.
Let~\mbox{$\hat{\Gamma}=\hat{\Gamma}_0\rtimes\hat{\Gamma}_1$}, where~\mbox{$\hat{\Gamma}_0\cong\HHH(\hat{r})$}
and $\hat{\Gamma}_1\cong\ZZ$,
and suppose that $\Gamma\subset\hat{\Gamma}$ is a normal subgroup.
Then the group $G=\hat{\Gamma}/\Gamma$ is finite and has a normal abelian subgroup
of index at most~$\nu$.
\end{lemma}
\begin{proof}
The group $G$ is finite for obvious reasons (cf. Lemma~\ref{lemma:Heisenberg-finite-index}).
By Lemma~\ref{lemma:Heisenberg-exercise}(iii) we have $\Gamma_0=\Gamma\cap\hat{\Gamma}_0$.
Thus by Lemma~\ref{lemma:Inoue-normal-subgroup}
there is a positive integer $k$ with the following properties: the subgroup $\Gamma_1\subset\Gamma$ is generated by an element $\hat{\upgamma}^k\hat{\updelta}$, where $\hat{\upgamma}$ is
a generator of $\hat{\Gamma}_1$, and $\hat{\updelta}$ is an element of $\hat{\Gamma}_0$;
and the group $G$ contains a normal subgroup $G_0\cong\hat{\Gamma}_0/\Gamma_0$
of index~$k$.
Note that the subgroup $\Gamma_0$ is normal in $\hat{\Gamma}$,
because $\Gamma$ is normal in $\hat{\Gamma}$,
while~$\Gamma_0$ is a characteristic subgroup of $\Gamma$ by Lemma~\ref{lemma:Heisenberg-exercise}(iii).
Let~\mbox{$R\in\GL_2(\ZZ)$} be the matrix that defines the semi-direct product~\mbox{$\hat{\Gamma}=\hat{\Gamma}_0\rtimes\hat{\Gamma}_1$}.
Considering the action of the element $\hat{\upgamma}$
on the lattice $\hat{\Gamma}_0/\z(\hat{\Gamma}_0)\cong\ZZ^2$ and its sublattice
$\Gamma_0/\z(\Gamma_0)\cong\ZZ^2$,
we see that $R^k$ is conjugate to $M$.
Thus $k$ is bounded by some constant that depends only on $M$
(that is, only on $\Gamma$) by Lemma~\ref{lemma:matrix-root}.
On the other hand, the group $G_0$ contains a normal abelian subgroup
of index at most $r$ by Corollary~\ref{corollary:Heisenberg-quotients},
and the assertion easily follows.
\end{proof}

\section{Heisenberg groups and direct products}
\label{section:Heisenberg-direct}

Consider the groups $\Gamma_0\cong \HHH(r)$ and~\mbox{$\Gamma_1\cong \ZZ$},
and put $\Gamma=\Gamma_0\times\Gamma_1$. One has
$$
\z(\Gamma)=\langle \updelta_3,\upgamma \rangle\cong\ZZ^2,
$$
and $\bar{\Gamma}=\Gamma/\z(\Gamma)\cong\ZZ^2$.

Unlike the situation in~\S\ref{section:lattices-semi-direct}
and~\S\ref{section:Heisenberg-semi-direct}, the subgroup
$\Gamma_0$ is not characteristic in $\Gamma$. Indeed,
let $\updelta_1$, $\updelta_2$, and $\updelta_3$ be the generators of
$\Gamma_0$ as in~\eqref{eq:Heisenberg-generators},
and $\upgamma$ be a generator of~$\Gamma_1$. Define an automorphism $\psi$ of $\Gamma$ by
$$
\psi(\updelta_1)=\updelta_1\upgamma, \
\psi(\updelta_2)=\updelta_2, \
\psi(\updelta_3)=\updelta_3, \
\psi(\upgamma)=\upgamma,
$$
cf.~\cite{Osipov}.
Then $\psi$ does not preserve $\Gamma_0$.
However, the following weaker uniqueness result holds.

\begin{lemma}\label{lemma:nearly-unique}
Let $\Gamma_0'\subset\Gamma$ be a normal subgroup isomorphic
to $\HHH(r')$ for some positive integer $r'$.
Suppose that $\varsigma\colon \Gamma/\Gamma_0'\cong\ZZ$. Then
the natural projection $\Gamma_0'\to\Gamma_0$ is an isomorphism.
In particular, one has $r=r'$.
\end{lemma}
\begin{proof}
Put $\Upsilon_1=\Gamma_1/\Gamma_0'\cap\Gamma_1$
and $\Upsilon_0=\Gamma_0/\varsigma(\Gamma_0')$. Then
$\Gamma/\Gamma_0'\cong\Upsilon_0\times\Upsilon_1$.
Therefore, either~$\Upsilon_0$ is trivial and $\Upsilon_1\cong\ZZ$,
or $\Upsilon_0\cong\ZZ$ and $\Upsilon_1$ is trivial.
In the former case $\varsigma$ provides an isomorphism
from $\Gamma_0'$ to $\Gamma_0$.
In the latter case $\Gamma_1\subset\Gamma_0'$ and the group $\varsigma(\Gamma_0')\cong\Gamma_0'/\Gamma_1$
is abelian by Lemma~\ref{lemma:Heisenberg-finite-index}. Thus the group
$\Gamma_0'$ is abelian as well, which is a contradiction.
\end{proof}

\begin{lemma}\label{lemma:Heisenberg-direct}
Suppose that $\Gamma$ is a normal subgroup
in a group~\mbox{$\hat{\Gamma}=\hat{\Gamma}_0\times\hat{\Gamma}_1$},
where~\mbox{$\hat{\Gamma}_0\cong\HHH(\hat{r})$}
and $\hat{\Gamma}_1\cong\ZZ$.
Then the group $G=\hat{\Gamma}/\Gamma$ is finite and has a normal abelian subgroup
of index at most~$r$.
\end{lemma}
\begin{proof}
The group $G$ is finite for obvious reasons.
Put $\Gamma_0'=\Gamma\cap\hat{\Gamma}_0$ and $G_0=\hat{\Gamma}_0/\Gamma_0'$.
By Lemma~\ref{lemma:Inoue-normal-subgroup}
one has $\Gamma=\Gamma_0' \rtimes \Gamma_1'$, where
$\Gamma_1'\cong\ZZ$ is generated by $\hat{\upgamma}^k\hat{\updelta}$ for some positive integer $k$, a generator $\hat{\upgamma}$
of $\hat{\Gamma}_1$, and an element $\hat{\updelta}\in \hat{\Gamma}_0$.
Thus $\Gamma/\Gamma_0'\cong\ZZ$. Since $\Gamma_0'$ is a subgroup of finite index in $\hat{\Gamma}_0$,
by Lemma~\ref{lemma:Heisenberg-finite-index} one has $\Gamma_0'\cong\HHH(r')$ for some $r'$.
By Lemma~\ref{lemma:nearly-unique} we know that $r'=r$. Thus
$G_0$ contains a normal abelian subgroup
$N$ of index at most $r$ by Corollary~\ref{corollary:Heisenberg-quotients}.
On the other hand, $G$ is generated by $G_0$ and the image $\bar{\upgamma}$ of~$\hat{\upgamma}$.
Since $\bar{\upgamma}$ is a central element in $G$, the group generated by $N$ and $\bar{\upgamma}$ is
a normal abelian subgroup of index at most~$r$ in~$G$.
\end{proof}

\section{Inoue surfaces}
\label{section:Inoue}

In this section we study automorphism groups of Inoue surfaces.

Inoue surfaces are quotients of $\CC\times \mathbb H$, where $\mathbb H$ is the upper
half-plane, by certain infinite discrete groups.
They were introduced by M.\,Inoue \cite{Inoue1974}.
These surfaces contain no curves and their invariants are as follows:
\[
\ad(X)=0,\quad \bb_1(X)=1, \quad \bb_2(X)=0,\quad \hh^{1,0}(X)=0,\quad \hh^{0,1}(X)=1.
\]

\begin{lemma}
\label{lemma-Inoue-quotient}
Let $X$ be an Inoue surface, and $G\subset \Aut(X)$ be a finite subgroup.
Then the action of $G$ on $X$ is free, and the quotient $\hat{X}=X/G$ is again
an Inoue surface.
\end{lemma}

\begin{proof}
Assume that the action of $G$ on $X$ is not free.
Let $g$ be an element of $G$ that acts on $X$ with fixed points.
To get a contradiction we may assume that the order of $g$ is prime.
Since $X$ contains no curves, the fixed point locus of $G$
consists of a finite number of points. Denote by $n$ the number of such points.
By the topological
Lefschetz fixed point formula, one has
\[
n=\sum\limits_{i=0}^4 (-1)^i\operatorname{tr}_{H^i(X,\RR)} g^*
= 2 - 2\operatorname{tr}_{H^1(X,\RR)} g^*.
\]
Hence the action of $g^*$ on $H^1(X,\RR)\cong\RR$ is not trivial.
This is possible only if $g$ is of order~$2$ and $n=4$.

Consider the quotient $\check{X}$ of $X$ by
the cyclic group generated by $g$. We see that $\check{X}$ has exactly $4$ singular points which are Du Val of type~$\mathrm{A}_1$.
Let~\mbox{$Y\to \check{X}$} be the minimal resolution of singularities.
Then
$$
\cc_1(Y)^2=\cc_1(\KKK_{\check{X}})^2=\frac 12 \cc_1(X)^2=0,
$$
and $\chit(Y)=4+\chit(\check{X})=6$.
This contradicts the Noether's formula, see e.g.~\cite[\S\,I.5]{BHPV-2004}.

Therefore, the action of $G$ on $X$ is free, and the quotient morphism $X\to \hat X=X/G$
is an unramified finite cover.
This implies that $\chit(\hat{X})=0$. Furthermore, one has
\[
\bb_2(\hat X)=\rk H^2(X,\ZZ)^G=0,
\]
and so $\bb_1(\hat X)=1$.
Therefore, by Enriques--Kodaira classification
$\hat X$ is a minimal surface of class VII.
Clearly, the surface $\hat X$ contains no curves, so that in particular $\bb_2(\hat{X})=0$. Thus
$\hat X$ is either a Hopf surface or an Inoue surface (see \cite{Bogomolov} and~\cite{Teleman-classification},
cf.~\cite{BBK}).
Since every Hopf surface contains a curve (see~\cite[Theorem~32]{Kodaira-structure-2}),
we conclude that~$\hat X$ is an Inoue surface.
\end{proof}

There are three types of Inoue surfaces:
$\mathrm{S_M}$, $\mathrm{S^{(+)}}$, and $\mathrm{S^{(-)}}$.
They are distinguished by the type
of their fundamental group $\Gamma=\pi_1(X)$, see~\cite{Inoue1974}:
\par\bigskip\noindent
{\rm
\setlength{\extrarowheight}{3pt}
\newcommand{\heading}[1]{\multicolumn{1}{c|}{#1}}
\newcommand{\headingl}[1]{\multicolumn{1}{c}{#1}}
\begin{tabularx}{0.97\textwidth}{p{0.06\textwidth}|p{0.1\textwidth}|p{0.7\textwidth}}
type & \heading{generators} & \multicolumn{1}{Y}{relations}
\\\Xhline{2\arrayrulewidth}
$\mathrm{S_M}$ & $\updelta_1, \updelta_2, \updelta_3, \upgamma$ &
$[\updelta_i,\updelta_j]=1$,
$\upgamma\updelta_i\upgamma^{-1}=\updelta_1^{m_{1,i}}\updelta_2^{m_{2,i}}\updelta_3^{m_{3,i}}$,
$(m_{j,i})\in \SL_3(\ZZ)$
\\\hline
$\mathrm{S^{(\pm)}}$ &
$\updelta_1, \updelta_2, \updelta_3, \upgamma$
&
$[\updelta_i,\updelta_3]=1$,
$[\updelta_1,\updelta_2]=\updelta_3^{r}$,
$\upgamma\updelta_i\upgamma^{-1}=\updelta_1^{m_{1,i}}\updelta_2^{m_{2,i}}\updelta_3^{p_i}$ for $i=1,2$,
$\upgamma\updelta_3\upgamma^{-1}=\updelta_3^{\pm 1}$,
$(m_{j,i})\in \GL_2(\ZZ)$, $\det (m_{j,i})=\pm 1$
\end{tabularx}
}
\par\bigskip
In the notation of \S\ref{section:lattices-semi-direct} and~\S\ref{section:Heisenberg-semi-direct},
one has $\Gamma\cong\Gamma_0\rtimes\Gamma_1$,
where $\Gamma_1\cong\ZZ$, while $\Gamma_0\cong\ZZ^3$
for Inoue surfaces of type $\mathrm{S_M}$, and $\Gamma_0\cong\HHH(r)$
for Inoue surfaces of types $\mathrm{S^{(\pm)}}$.
In the former case the matrix $M\in\SL_3(\ZZ)$ that defines the semi-direct product has
eigenvalues $\upalpha$, $\upbeta$, and $\bar \upbeta$, where $\upalpha\in \RR$,
$\upalpha>1$, and~\mbox{$\upbeta\notin\RR$}.
In the latter case the matrix $M\in\GL_2(\ZZ)$
that defines the action of~$\ZZ$ on~\mbox{$\HHH(r)/\z(\HHH(r))\cong\ZZ^2$} has
real eigenvalues $\upalpha$ and $\upbeta$,
where~\mbox{$\upalpha>1$} and $\upalpha\upbeta=\pm1$ depending on whether $\Gamma$ is of type~$\mathrm{S^{(+)}}$ or~$\mathrm{S^{(-)}}$,
see~\mbox{\cite[\S\S2--4]{Inoue1974}}.

\begin{lemma}\label{lemma:Inoue-group-types}
Let $\Gamma$ be a group of one of the types $\mathrm{S_M}$, $\mathrm{S^{(+)}}$, or $\mathrm{S^{(-)}}$.
Then
\begin{enumerate}
\item
$\Gamma$ is of type $\mathrm{S_M}$ if and only if $\Gamma$ contains a characteristic subgroup
isomorphic to~$\ZZ^3$;
\item
$\Gamma$ is of type $\mathrm{S^{(+)}}$ if and only if
$\Gamma$ contains no subgroups isomorphic to $\ZZ^3$ and~\mbox{$\z(\Gamma)\neq\{1\}$};
\item
$\Gamma$ is of type $\mathrm{S^{(-)}}$ if and only if $\Gamma$ contains no subgroups isomorphic to $\ZZ^3$
and~\mbox{$\z(\Gamma)=\{1\}$}.
\end{enumerate}
\end{lemma}
\begin{proof}
This follows from Lemmas~\ref{lemma:lattice-exercise}(ii)
and~\ref{lemma:Heisenberg-semi-direct-center}(ii)
and Remark~\ref{remark:Heisenberg-semi-direct-no-1}.
\end{proof}

\begin{corollary}
\label{corollary-Inoue-quotient-types}
Let $X$ be an Inoue surface, and $G\subset \Aut(X)$ be a finite subgroup.
Then the action of $G$ on $X$ is free, and the following assertions hold.
\begin{enumerate}
\item
If $X$ is of type $\mathrm{S_M}$, then so is $X/G$;

\item
If $X$ is of type $\mathrm{S^{(-)}}$, then so is $X/G$;

\item
If $X$ is of type $\mathrm{S^{(+)}}$, then $X/G$ is of type $\mathrm{S^{(+)}}$ or $\mathrm{S^{(-)}}$.
\end{enumerate}
\end{corollary}
\begin{proof}
Put $\hat{X}=X/G$. Then the action of $G$ on $X$ is free, and $X$ is an
Inoue surface by Lemma~\ref{lemma-Inoue-quotient}. Put
$\hat{\Gamma}=\pi_1(\hat{X})$. Then $\hat{\Gamma}$ is a group of one of the
types $\mathrm{S_M}$, $\mathrm{S^{(+)}}$, or $\mathrm{S^{(-)}}$, and~\mbox{$\Gamma\subset\hat{\Gamma}$} is a normal subgroup of finite index. Now everything
follows from Lemma~\ref{lemma:Inoue-group-types}.
\end{proof}

\begin{lemma}\label{lemma:Jordan-Inoue-M}
Let $X$ be an Inoue surface of type $\mathrm{S_M}$. Then the group $\Aut(X)$ is Jordan.
\end{lemma}
\begin{proof}
Let $G\subset \Aut(X)$ be a finite subgroup, and put $\hat{X}=X/G$.
By Corollary~\ref{corollary-Inoue-quotient-types}
the action of $G$ on $X$ is free, and $\hat{X}$ is also an Inoue surface of type $\mathrm{S_M}$.
Put $\Gamma=\pi_1(X)$
and~\mbox{$\hat{\Gamma}=\pi_1(\hat{X})$}. Then
$\Gamma$ is a normal subgroup of $\hat{\Gamma}$, and $\hat{\Gamma}/\Gamma\cong G$;
moreover, both~$\Gamma$ and~$\hat{\Gamma}$ are semi-direct products
as in~\S\ref{section:lattices-semi-direct}.
Now it follows from Lemma~\ref{lemma:lattices-semi-direct}
that there is a constant $\nu$ that depends only on $\Gamma$ (that is, only on $X$),
such that $G$ has a normal abelian subgroup of index at most~$\nu$.
\end{proof}

Now we deal with Inoue surfaces of types $\mathrm{S^{(+)}}$ and $\mathrm{S^{(-)}}$.

\begin{lemma}\label{lemma:Jordan-Inoue-pm}
Let $X$ be an Inoue surface of type $\mathrm{S^{(+)}}$ or $\mathrm{S^{(-)}}$.
Then the group $\Aut(X)$ is Jordan.
\end{lemma}
\begin{proof}
Let $G\subset \Aut(X)$ be a finite subgroup, and put $\hat{X}=X/G$.
By Corollary~\ref{corollary-Inoue-quotient-types}
the action of $G$ on $X$ is free, and $\hat{X}$ is also an Inoue surface of type
$\mathrm{S^{(+)}}$ or $\mathrm{S^{(-)}}$.
Put~\mbox{$\Gamma=\pi_1(X)$}
and $\hat{\Gamma}=\pi_1(\hat{X})$. Then
$\Gamma$ is a normal subgroup of $\hat{\Gamma}$, and $\hat{\Gamma}/\Gamma\cong G$;
moreover, both~$\Gamma$ and~$\hat{\Gamma}$ are semi-direct products
as in~\S\ref{section:Heisenberg-semi-direct}.
Now it follows from Lemma~\ref{lemma:Heisenberg-semi-direct}
that there is a constant $\nu$ that depends only on $\Gamma$ (that is, only on $X$),
such that $G$ has a normal abelian subgroup of index at most~$\nu$.
\end{proof}

We summarize the results of Lemmas~\ref{lemma:Jordan-Inoue-M}
and~\ref{lemma:Jordan-Inoue-pm}
as follows.

\begin{corollary}\label{corollary:Inoue}
Let $X$ be an Inoue surface. Then the group $\Aut(X)$ is Jordan.
\end{corollary}

\begin{remark}
There are certain types of minimal compact complex surfaces
of class~VII whose automorphism groups are studied in details, for instance, hyperbolic and parabolic Inoue surfaces,
see~\cite{Pinkham} and~\cite{Fujiki-ParabolicInoue}.
Note that surfaces of both of these types have positive second Betti numbers
(and thus they are not to be confused with Inoue surfaces we deal with in this section).
\end{remark}

\section{Kodaira surfaces}
\label{section:Kodaira}

In this section we study automorphism groups of Kodaira surfaces.
Our approach here is similar to what happens in~\S\ref{section:Inoue}.

Recall (see e.g. \cite[\S\,V.5]{BHPV-2004}) that a Kodaira surface
is a compact complex surface of Kodaira dimension $0$ with odd first Betti number.
There are two types of Kodaira surfaces: primary and secondary ones.
A primary Kodaira surface
is a compact complex surface with
the following invariants \cite[\S6]{Kodaira-structure-1}:
\[
\newcommand{\cquad}{\hspace{4pt}}
\KKK_X \sim 0, \cquad
\ad(X)=1,\cquad
\bb_1(X)=3,\cquad \bb_2(X)=4,\cquad \chit(X)=0,\cquad \hh^{0,1}(X)=2,\cquad
\hh^{0,2}(X)=1.
\]
A secondary Kodaira surface is a quotient of a primary Kodaira surface by a
free action of a finite cyclic group.

Let $X$ be a primary Kodaira surface.
The universal cover of $X$ is isomorphic to $\CC^2$, and the fundamental group $\Gamma=\pi_1(X)$
has the following presentation:
\begin{equation}
\label{equation-fundamenta-group-Kodaira-surface}
\Gamma=\langle \updelta_1,\updelta_2,\updelta_3,\upgamma
\mid [\updelta_1, \updelta_2]= \updelta_3^r,\
[\updelta_i, \updelta_3]=[\updelta_i,\upgamma]=1
\rangle,
\end{equation}
where $r$ is a positive integer \cite[\S6]{Kodaira-structure-1}.
In the notation of \S\ref{section:Heisenberg-direct} one has~\mbox{$\Gamma\cong\HHH(r)\times\ZZ$}.

Denote by $\overline{\Aut}(X)\subset \Aut (X)$ the subgroup that consists
of all elements acting trivially on $H^*(X,\QQ)$ and $H^*(B,\QQ)$.

\begin{lemma}[{cf. Lemma~\ref{lemma-Inoue-quotient}}]
\label{lemma:again-Kodaira}
Let $X$ be a primary Kodaira surface, and
$G\subset \overline{\Aut}(X)$ be a finite subgroup.
Then the action of $G$ on $X$ is free, and the quotient $\hat{X}=X/G$
is again a primary Kodaira surface.
\end{lemma}
\begin{proof}
Let $\phi\colon X\to B$ be the algebraic reduction of $X$.
Then~$B$ is an elliptic curve, and
$\phi$ is a principal elliptic fibration~\mbox{\cite[\S6]{Kodaira-structure-1}}, \mbox{\cite[\S\,V.5]{BHPV-2004}}.
Furthermore, $\phi$ is equivariant with respect to~\mbox{$\Aut(X)$}.

Let $g$ be an element of $G$.
Since the curve $B$ is elliptic, we see that $g$ acts on $B$ without fixed points.
This means that there are no fibers of $\phi$ that consist of points fixed by $g$.
On the other hand, every curve on $X$ is a fiber of $\phi\colon X\to B$. Indeed, otherwise
one can construct a curve on $X$ with a positive self-intersection, which would
imply that $X$ is projective, see~\mbox{\cite[\S\,IV.6]{BHPV-2004}}.
Hence there are no curves that consist of points fixed by $g$ on $X$ at all.
Now the topological Lefschetz fixed point formula shows that
the number of points on $X$ fixed by $g$ equals
$$
\sum\limits_{i=0}^4 (-1)^i\operatorname{tr}_{H^i(X,\RR)} g^*
= \sum\limits_{i=0}^4 \bb_i(X)=0.
$$

Therefore, the action of $G$ on $X$ is free, so that $\hat{X}$ is a smooth surface and the quotient morphism
$X\to \hat{X}$ is an unramified finite cover. Hence
$\varkappa(\hat{X})=\varkappa(X)=0$. Moreover, we have
\[
\cc_1(\hat{X})^2=\cc_1(X)^2=0.
\]
This means that the surface $\hat{X}$ is minimal.
Since $G\subset \overline{\Aut}(X)$, we have
\[
\bb_1(\hat{X})=\bb_1(X)=3.
\]
Therefore, $\hat{X}$ is a primary Kodaira surface
by Kodaira--Enriques classification.
\end{proof}

\begin{lemma}
\label{lemma:primary-Kodaira-surface}
Let $X$ be a primary Kodaira surface.
Then the group $\Aut(X)$ is Jordan.
\end{lemma}
\begin{proof}
Let $G\subset \Aut(X)$ be a finite subgroup.
By Theorem~\ref{theorem:Minkowski} and Lemma~\ref{lemma:group-theory}
we can assume that~\mbox{$G\subset \overline{\Aut}(X)$}.
Put $\hat{X}=X/G$.
It follows from Lemma~\ref{lemma:again-Kodaira}
that $G$ acts freely on~$X$, and $\hat{X}$ is a primary Kodaira surface.
Put $\Gamma=\pi_1(X)$
and $\hat{\Gamma}=\pi_1(\hat{X})$. Then
$\Gamma$ is a normal subgroup of $\hat{\Gamma}$, and $\hat{\Gamma}/\Gamma\cong G$;
moreover, both $\Gamma$ and $\hat{\Gamma}$ are
as in~\S\ref{section:Heisenberg-direct}.
Now it follows from Lemma~\ref{lemma:Heisenberg-direct}
that there is a constant $r$ that depends only on $\Gamma$ (that is, only on~$X$),
such that $G$ has a normal abelian subgroup of index at most~$r$.
\end{proof}

%\bibliography{non-proj}
%\bibliographystyle{alpha}

\end{document}